\documentclass[12pt,a4paper,twoside]{article}

\usepackage{amsmath,amssymb,amsfonts,amsthm}
\usepackage{times,hyperref,color}
\usepackage{graphicx}

\let\mathcal\mathscr
\newtheorem{theorem}{\bf Theorem}[section]
\newtheorem{proposition}[theorem]{\bf Proposition}
\newtheorem{lemma}[theorem]{\bf Lemma}
\newtheorem{corollary}[theorem]{\bf Corollary}

\newtheorem{definition}[theorem]{\bf Definition}

\usepackage{epsfig}
\begin{document}

\title{On the star-critical Ramsey number of a forest versus complete graphs}
\author{Azam Kamranian, Ghaffar Raeisi$^{\small 1}$\vspace{.6 cm}\\
{\footnotesize Department of Mathematical Sciences, Shahrekord University, Shahrekord, P.O.Box 115, Iran}\\
{\footnotesize azamkamranian@stu.sku.ac.ir,~~ g.raeisi@sci.sku.ac.ir }\\}

\date{ }

{\footnotesize\medskip\maketitle\footnotetext[1] {Corresponding author}}

\begin{abstract}
{\footnotesize
Let $G$ and $G_1, G_2, \ldots , G_t$ be given graphs. By $G\rightarrow (G_1, G_2, \ldots , G_t)$ we mean if the edges of $G$ are arbitrarily colored by $t$ colors, then for some $i$, $1\leq i\leq t$, the spanning subgraph of $G$ whose edges are colored with the $i$-th color, contains a copy of $G_i$. The Ramsey number $R(G_1, G_2, \ldots, G_t)$ is the smallest positive integer $n$ such that $K_n\rightarrow (G_1, G_2, \ldots , G_t)$ and the size Ramsey number $\hat{R}(G_1, G_2, \ldots , G_t)$ is defined as  $\min\{|E(G)|:~G\rightarrow (G_1, G_2, \ldots , G_t)\}$. Also, for given graphs $G_1, G_2, \ldots , G_t$ with $r=R(G_1, G_2, \ldots , G_t)$, the star-critical Ramsey number $R_*(G_1, G_2, \ldots , G_t)$ is defined as  $\min\{\delta(G):~G\subseteq K_r, ~G\rightarrow (G_1, G_2, \ldots , G_t)\}$. In this paper, the Ramsey number and also the star-critical Ramsey number of a forest versus any number of complete graphs will be computed exactly in terms of the Ramsey number of complete graphs. As a result, the computed star-critical Ramsey number is used to give a tight bound for the size Ramsey number of a forest versus a complete graph.}

\end{abstract}

\bigskip
\section{Introduction}
In this paper, we are only concerned with undirected simple finite graphs and we follow \cite{Boundy} for terminology and notation not defined here. For a given graph $G$, we denote its vertex set, edge set, maximum degree and minimum degree of $G$ by $V(G)$, $E(G)$, $\Delta(G)$ and $\delta(G)$, respectively. Also, for given disjoint subsets $U$ and $W$ of $V(G)$, we use $E[U,W]$ to denote the set of all edges between $U$ and $W$ and we use $G[U]$ to denote the subgraph of $G$ induced by the vertices of $U$. A {\it clique} in a graph is a set of mutually adjacent vertices and  the maximum size of a clique in a graph $G$ is called the {\it clique number} of $G$. As usual, the star graph on $n+1$ vertices is denoted by $K_{1,n}$ and the complete graph on $n$ vertices is denoted by $K_n$.  In this paper, we use $K_k(n_1,\ldots ,n_k)$ to denote the complete $k$-partite graph in which the $i$-th part, $1\leq i\leq k$,  has $n_i$ vertices. In addition, for a given red/blue coloring of the edges of a graph $G$, we use $G^r$ and $G^b$ to denote the spanning subgraphs of $G$ induced by the edges of colors red and blue, respectively.

\medskip
A {\it tree} is a connected graph without cycles and a {\it forest} is a disjoint union of trees.  In this paper, for a given forest $F$, we use $n(F)$ to denote the number of vertices of the largest component of $F$ and for $i=1,2,\ldots, n(F)$, we use $k_i(F)$ to denote the number of components of $F$ with exactly $i$ vertices. Moreover, for a given forest $F$, we define  $C(F)=\{i:~k_i(F)\not=0\}$ and the {\it variety} of $F$, denoted by $q(F)$, is the number of components of $F$ with different sizes, i.e. $q(F)=|C(F)|$.

\bigskip
For given graphs $G$ and $G_1, G_2, \ldots , G_t$, we write $G\rightarrow (G_1, G_2, \ldots , G_t)$ if, when the edges of $G$ are colored in any fashion with $t$ colors, then for some $i$, $1\leq i\leq t$,  the spanning subgraph of $G$ whose edges are colored with the $i$-th color, contains a copy of $G_i$. For given graphs $G_1, G_2, \ldots , G_t$, the {\it Ramsey number} $R(G_1, G_2, \ldots, G_t)$ is the smallest positive integer $n$ such that $K_n\rightarrow (G_1, G_2, \ldots, G_t)$. The existence of such a positive integer is guaranteed by Ramsey's classical result \cite{Ramsey}. For a survey on Ramsey theory and results in this area, we refer the reader to the regularly updated survey by Radziszowski \cite{survey}. A $t$-tuple $(G_1, G_2, \ldots , G_t)$ of graphs with $r=R(G_1, G_2, \ldots , G_t)$,  is called {\it Ramsey-full} if $K_r\rightarrow(G_1, G_2, \ldots , G_t)$, but $K_r-e\nrightarrow(G_1, G_2, \ldots , G_t)$, for each $e\in E(K_r)$. Also, a $t$-coloring of the edges of $K_n$ is called a $(G_1, G_2, \ldots , G_t)$-{\it free coloring} if for each $i$, $1\leq i\leq t$,  the spanning subgraph of $G$ whose edges are colored with the $i$-th color, does not contain $G_i$ as a subgraph.

\medskip
Classical Ramsey numbers may be thought of in a more general setting. For given graphs $G_1, G_2, \ldots , G_t$ and any monotone graph parameter $\rho$, one may define the {\it $\rho$-Ramsey number}, denoted by $R_{\rho}(G_1, G_2, \ldots , G_t)$, to be
$\min\{\rho(H):~H\rightarrow(G_1, G_2, \ldots, G_t)\}.$ Note that $R_{\rho}(G_1, G_2, \ldots , G_t)$ is the classical Ramsey number $R(G_1, G_2, \ldots , G_t)$ when $\rho(H)$ is the number of vertices of $H$. The study of parameter Ramsey numbers dates back to the 1970s \cite{SBE} and  since then, many researchers have
studied this quantity when $\rho(H)$ is the clique number of $H$ \cite{12,23,25} (giving way to the study of Folkman numbers), when
$\rho(H)$ is the number of edges in $H$ \cite{4,9,10,s10,s11,27} (called the size Ramsey number), when $\rho(H)$ is the chromatic number of $H$ \cite{SBE,28} (called the chromatic Ramsey number) or when $\rho(H)$ is the maximum degree of $H$ \cite{SBE,14,16,17,18} (called the degree Ramsey number).

\medskip
In this paper, we are interested in the size Ramsey number ($\rho(H)=|E(H)|$) and the  star-critical Ramsey number ($\rho(H)=\delta(H)$ and $V(H)=R(G_1, G_2, \ldots , G_t)$). For given graphs $G_1, G_2, \ldots , G_t$ with $r=R(G_1, \ldots , G_t)$,  the {\it star-critical Ramsey number}
 $R_*(G_1, G_2, \ldots, G_t)$ is defined as 
$\min\{\delta(H):~ H\subseteq K_{r},~ H\rightarrow(G_1, G_2, \ldots , G_t)\},$ and the {\it size Ramsey number} $\hat{R}(G_1, G_2, \ldots , G_t)$ is defined as  the minimum number of edges of a graph $H$ such that $H\rightarrow (G_1, G_2, \ldots , G_t)$.

\medskip
Let $K_{n}\sqcup K_{1,k}$ be the graph obtained from $K_n$ by adding a new vertex $v$ adjacent to $k$ vertices of $K_{n}$. It is easy to see that the star-critical Ramsey number $R_*(G_1, G_2, \ldots , G_t)$ is equivalent to finding the smallest integer $k$ such that $K_{r-1}\sqcup K_{1,k}\rightarrow(G_1, G_2, \ldots , G_t)$, where $r=R(G_1, G_2, \ldots , G_t)$. Size and star-critical Ramsey numbers were investigated by several authors (see \cite{4,9,10,s9,s10,s11,hoka,s18,27}). The exact value of the star-critical Ramsey number for a tree versus a complete graph was determined in \cite{hoka} and the star-critical Ramsey number of a matching (as a special case of a forest) versus a complete graph was determined in \cite{s10,s18}. In this paper, the Ramsey number and also the star-critical Ramsey number of a forest versus any number of complete graphs will be determined exactly in terms of the Ramsey number of complete graphs. More precisely, if $m_1, m_2, \ldots, m_t$ are positive integers, $r=R(K_{m_1},\ldots,K_{m_t})$ and $F$ is a forest without isolated vertices, then
$$R(F,K_{m_1},\ldots,K_{m_t})=\max_{j\in C(F)} \{(j-1)(r-2)+ \sum_{i=j}^{n(F)} ik_i(F) \}.$$ Moreover, if $j_0$ is the smallest value of $j$ that realizes ${\max_{j\in C(F)}} \{(j-1)(r-2)+ \sum_{i=j}^{n(F)}ik_i(F)\}$, then
$$R_*(F, K_{m_1},\ldots,K_{m_t})=(j_0-1)(r-3)+ \sum_{i=j_0}^{n(F)}ik_i(F).$$ In addition, the star-critical Ramsey number $R_*(F,K_{m})$ is used to give a tight bound for the size Ramsey number of a forest versus a complete graph.

\section{Main Results}

\hspace{.6 cm}To determine the exact value of the star-critical Ramsey number of a forest versus complete graphs, we need some theorems and lemmas. First, we start with the following theorem \cite{S}, giving the exact value of the Ramsey number of a forest versus a complete graph.

\begin{theorem}{\rm (\cite{S})}\label{Rforest}
If $F$ is an arbitrary forest, then
$$R(F, K_m)=\max_{j\in C(F)} \{(j-1)(m-2)+ \sum_{i=j}^{n(F)} ik_i(F) \}.$$
\end{theorem}

As an easy and immediate consequence of Theorem \ref{Rforest}, we have the following corollary.

\begin{corollary}\label{ramseyktreecomplete}
If $T_n$ is an arbitrary tree on $n$ vertices, then $R(kT_n, K_m)=(n-1)(m-2)+nk.$
\end{corollary}

\medskip
In addition, the exact value of the star-critical Ramsey number for a tree versus a complete graph was determined in \cite{hoka}. In fact, in \cite{hoka} it was proved that for a given tree $T_n$, all $(T_n,K_m)$-free colorings of the complete graph on  $R(T_n,K_m)-1$ vertices are unique.

\begin{theorem}{\rm (\cite{hoka})}\label{staruniq}
For given $n$ and $m$ with $n,m\geq 2$, let $r = R(T_n, K_m) = (n-1)(m-1) + 1$. If $c$ is a $(T_n, K_m)$-free coloring of $G=K_{r-1}$, then the resulting graph must have a unique red/blue coloring as follows.
\begin{eqnarray*}
&G^r&=(m-1)K_{n-1}\\
&G^b&=K_{m-1}(n-1, n-1,\ldots ,n-1).
\end{eqnarray*}
\end{theorem}

\medskip
\begin{corollary}{\rm (\cite{hoka})}\label{startreecomplete}
For any tree $T_n$ on $n$ vertices, $R_*(T_n ,K_m) =(n-1)(m-2)+1.$
\end{corollary}

\medskip
Now, in the sequel, we determine the star-critical Ramsey number of a disjoint union of trees versus a complete graph. For this purpose, first we characterize the class of all $(kT_n,K_m)$-free colorings on $R(kT_n,K_m)-1$ vertices and then we use such a characterization to find the star-critical Ramsey number $R_*(kT_n,K_m)$.

\begin{definition}\label{uni}
Let $n\geq 2$, $m\geq 3$ be given integers and $r = R(kT_n, K_m) = (n-1)(m-2)+nk$. Define the family $\mathbb{G}_{n,m,k}$ as $\{G_0, G_1, \ldots, G_{k-1}\}$, where each $G_i$, $0\leq i\leq k-1$, is a copy of $K_{r-1}$ with the following red/blue edge coloring:
\begin{eqnarray*}
G_i:~ &G_i^r&=K_{(n-1)+nk_{_1}}\cup K_{(n-1)+nk_{_2}}\cup\ldots\cup K_{(n-1)+nk_{_{m-2}}}\cup K_{(k-i)n-1}\\
&G_i^b&=K_{m-1}((n-1)+nk_{_1}, (n-1)+nk_{_2},\ldots ,(n-1)+nk_{_{m-2}}, (k-i)n-1),
\end{eqnarray*}
where, $k_1,k_2,\ldots,k_{m-2}$ are non-negative integers with $\sum_{j=1}^{m-2} k_{_j}=i$. Such a red/blue coloring of $G_i\in \mathbb{G}_{n,m,k}$ is shown in Figure \ref{uniqpic}.
\begin{figure}[h!]
\includegraphics[width=\textwidth]{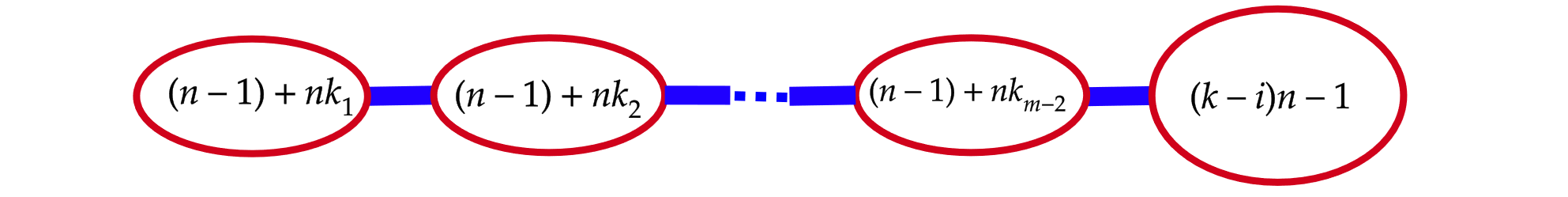}\caption{Red/blue coloring of $G_i\in \mathbb{G}_{n,m,k}$, described in Definition \ref{uni}.}\label{uniqpic}
\end{figure}
\end{definition}

\begin{lemma}\label{uniq}
Let $n\geq 2$, $m\geq 3$ and $k\geq 1$ be given and $r = R(kT_n, K_m) = (n-1)(m-2)+nk$. If $c$ is a $(kT_n, K_m)$-free coloring of $K_{r-1}$, then the resulting graph must belong to the family $\mathbb{G}_{n,m,k}$ described in Definition \ref{uni}.
\end{lemma}

\begin{proof}
We use induction on $k$ to prove the theorem. Let $c$ be an arbitrary $(kT_n, K_m)$-free coloring of $G=K_{r-1}$. If $k=1$, then by Theorem \ref{staruniq}, $c$ must be the red/blue coloring of $G_0$ described in Definition \ref{uni}. Thus, let $k\geq 2$. As $k\geq 2$ and $R(kT_n,K_m)-1\geq R(T_n, K_m)$, such a coloring must contain a red monochromatic copy of $T_n$ (say $T$). Delete the vertices of $T$ from $G$ and let $H$ be the resulting graph. Clearly,
$$|V(H)| = R(kT_n,K_m)-1-|V(T)| = (n-1)(m-2)+(k-1)n-1.$$
Since $|V(H)|=R((k-1)T_n,K_m)-1$ and the coloring induced by $c$ on $H$, say $c'$, is a $((k-1)T_n,K_m)$-free coloring, then  by the induction hypothesis, $c'$ must be a red/blue coloring of $H$ belonging to the family $\mathbb{G}_{n,m,k-1}$, say $H_i$, such that for some non-negative integers $k_1,k_2,\ldots,k_{m-2}$ with $\sum_{j=1}^{m-2} k_{_j}=i$,
\begin{eqnarray*}
H_i:~&H_i^r&=K_{(n-1)+nk_{_1}}\cup K_{(n-1)+nk_{_2}}\cup\ldots\cup K_{(n-1)+nk_{_{m-2}}}\cup K_{(k-i-1)n-1}\\
&H_i^b&=K_{m-1}((n-1)+nk_{_1}, (n-1)+nk_{_2},\ldots ,(n-1)+nk_{_{m-2}}, (k-i-1)n-1).
\end{eqnarray*}

\noindent Let $V_1, V_2,\ldots , V_{m-1}$ be the partite sets of $H_i^b$ such that for each $l$, $1\leq l\leq m-2$, $|V_l|=(n-1)+nk_l$ and $|V_{m-1}|=(k-1-i)n-1$. Note that for each $l$, $1\leq l\leq m-1$, $G^r[V_l]=K_{|V_l|}$. Now, consider an arbitrary vertex $v\in V(T)$ and set $V'=V(T)\setminus \{v\}$. We consider the following cases.

\bigskip
\noindent{\bf Case 1:} $E[v,V_{m-1}]\subseteq E(G^r)$.

\medskip
In this case,  $E[V',V_l]\subseteq E(G^b)$, for each $l=1,\ldots ,m-2$. Indeed, if there is a red edge $v'v_l\in E[V',V_l]$, for some $l$, $1\leq l\leq m-2$, then $\{v\}\cup V_{m-1}$ induces a copy of $K_{(k-i-1)n}$ in $G^r$ and $V_l\cup \{v'\}$ contains $k_l+1$ red copies of $T_n$. Therefore, the red subgraph of $G$ induced by $V_1\cup\ldots\cup V_{m-2}$ contains  $k_1+k_2+\ldots+(k_l+1)+\ldots+k_{m-2}=i+1$ copies of $T_n$ and so these copies together with the $(k-i-1)$ copies of $T$ in $G^r[\{v\}\cup V_{m-1}]$ form a $kT_n\subseteq G^r$, a contradiction.

\medskip
It is easy to see that $E[V',V_{m-1}]\subseteq E(G^r)$, otherwise, there is a blue copy of $K_m$ in $G$ by choosing an arbitrary vertex $v_j\in V_j$, for each $j$,  $1\leq j\leq m-2$, and the end vertices of the blue edge in $E[V',V_{m-1}]$. Now, by a similar argument as the first paragraph, we conclude that $E[v, V_l]\subseteq E(G^b)$, for each $l$, $1\leq l\leq m-2$.

\medskip
In addition, $G[V(T)]\subseteq G^r$, otherwise, if $u,u'\in V(T)$ and $uu'\in E(G^b)$, then the induced subgraph $G[u,u',V_1,\ldots,V_{m-2}]$ contains a blue monochromatic copy of $K_m$, which is impossible. Therefore, $G[V_l]$, $1\leq l\leq m-2$, and $G[V(T)\cup V_{m-1}]$ are subgraphs of $G^r$ and $|V(T)\cup V_{m-1}|=(k-i)n-1$. Therefore, the resulting graph is isomorphic to  $G_{i}\in\mathbb{G}_{n,m,k}$.

\bigskip
\noindent{\bf Case 2:} $uv\in E(G^b)$, for some $u\in V_{m-1}$.

In this case, the vertex $v$ cannot have a blue neighbour in each $V_l$, for $l=1,\ldots ,m-2$, otherwise, if $u_i$ is a blue neighbor of $v$ in $V_i$, $1\leq i\leq m-2$, then $G[v,u,u_1,\ldots,u_{m-2}]$ form a blue copy of $K_m$, a contradiction. So, without loss of generality, we may assume that $E[v,V_1]\subseteq E(G^r)$. If there is a red edge in $E[V',V_l]$, for some $l$, $2\leq l \leq m-1$, say $v'v_l$, then $(V_1\cup v)\cup\ldots\cup (V_l\cup v')\cup\ldots\cup V_{m-1}$ contains $k$ disjoint red copies of $T_n$, a contradiction. Therefore, $E[V',V_l]\subseteq E(G^b)$, for each  $l$,  $2\leq l\leq m-1$.  Now, since $K_m\nsubseteq G^b$, then $G[V']\subseteq G^r$ and $E[V',V_1]\subseteq E(G^r)$.   Also, since $kT_n\nsubseteq G^r$, then $E[v,V_l]\subseteq E(G^b)$, $l=2,\ldots ,m-1$ and so, $E[v,V']\subseteq E(G^r)$. Hence, $G[V(T)]\subseteq G^r$. Therefore,  $G[V(T)\cup V_{1}]$ and  $G[V_l]$, $2\leq l\leq m-1$ are subgraphs of $G^r$ and so, the resulting graph is isomorphic to  $G_{i-1}\in\mathbb{G}_{n,m,k}$,  which completes the proof.
\end{proof}

\begin{lemma}\label{starktreecomplete}
For every $n\geq 2$ and $m\geq 3$, $R_*(kT_n, K_m)=(n-1)(m-3)+nk$.
\end{lemma}

\begin{proof}
Let $r = R(kT_n, K_m) = (n-1)(m-2)+nk$ and $r_*$ be the claimed number for $R_*(kT_n, K_m)$. Let  $H=K_{r-1}\sqcup K_{1,r_*-1}$ and partition the vertices of $K_{r-1}$ into $(m-1)$ parts $V_1, V_2,\ldots,V_{m-1}$ such that for every $i$, $1\leq i\leq m-2$, $|V_i|=n-1$  and $|V_{m-1}|=nk-1$. Color all edges contained in $V_i$, $1\leq i\leq m-1$, by red and the rest by blue. Now, add a vertex $v$ with all blue edges adjacent to every vertex in $V_i$, $2\leq i\leq m-1$. Since the chromatic number of the blue subgraph of $H$ is $m-1$, then the blue graph do not contain $K_m$ as a subgraph. In the red subgraph, the components of size $(n-1)$ does not contain any red tree $T_n$ and in the component of size $(nk-1)$ there are at most $(k-1)$ red copies of $T_n$. Therefore, we have a $(kT_n,K_m)$-free coloring of $H$ and so $H\nrightarrow(kT_n,K_m)$.

\medskip
For the upper bound, let $G=K_{r-1}\sqcup K_{1,r_*}$ and let $v$ be the vertex of degree $r_*$ in $G$. We prove that $G\rightarrow(kT_n,K_m)$. On the contrary, let $G\nrightarrow(kT_n,K_m)$ and consider a $(kT_n,K_m)$-free coloring of $G$. This coloring induces a  $(kT_n,K_m)$-free  coloring of $G\setminus\{v\}\cong K_{r-1}$ and so, by Lemma \ref{uniq}, this coloring have the form specified as described in Definition \ref{uni}. Since the blue subgraph of $G\setminus\{v\}$ is a $(m-1)$-partite graph,  let $V_1, \ldots ,V_{m-2}, V_{m-1}$ be the partite sets of $G\setminus\{v\}$. Note that by the structure of the $(kT_n,K_m)$-free  coloring of  $G\setminus\{v\}$, all $r_*=(n-1)(m-3)+nk$ edges incident with $v$ must be blue, otherwise, $G$ contains a red $kT_n$, a contradiction.  As $\deg(v) \geq r_*$, one can easily conclude that $v$ has at least one neighbor in each  $V_i$, $1\leq i\leq m-1$, and hence, $G$ contains a blue copy of $K_m$, a contradiction. This contradiction shows that  $G\rightarrow(kT_n,K_m)$, which completes the proof.
\end{proof}

 


Now, we are ready to prove our first main result. In the following, we determine the star-critical Ramsey number of an arbitrary forest versus a complete graph. Hereafter, for simplicity, for a given forest $F$ and also positive integers $a$ and $b$, we use $\sum_{a}^{b}$ to denote $\sum_{i=a}^{b} ik_i(F)$ and briefly, we use $n$ and $k_i$ to denote $n(F)$ and $k_i(F)$, respectively.

\begin{theorem}\label{starforestcomplete}
Let $F$ be an arbitrary forest without isolated vertices, $m\geq 3$ and let $j_0$ be the smallest value of $j$ that realizes $\max_{j\in C(F)} \{(j-1)(m-2)+ \sum_{i=j}^{n(F)}ik_i(F)\}$. Then,

$$ R_*(F, K_m)=(j_0-1)(m-3)+ \displaystyle\sum_{i=j_0}^{n(F)}ik_i(F).$$
\end{theorem}

\begin{proof}
Let $r=R(F,K_m)=(j_0-1)(m-2)+ \sum_{j_0}^{n}$ and $r_*$ be the claimed number for $R_*(F, K_m)$. To see $r_*$ is a lower bound for $R_*(F, K_m)$, we show that if $H=K_{r-1}\sqcup K_{1,r_*-1}$ then $H\nrightarrow(F,K_m)$. Partition the vertices of $K_{r-1}$ into $(m-1)$ parts $V_1,\ldots,V_{m-1}$ such that for every $i$, $1\leq i\leq m-2$, $|V_i|=j_0-1$ and $|V_{m-1}|=\sum_{j_0}^{n}-1$. Color all edges with both ends in $V_i$, $1\leq i\leq m-1$, by red and the rest by blue. Now, add a vertex $v$ with all blue edges adjacent to every vertex in $V_i$, $2\leq i\leq m-1$. Clearly, the clique number of the  subgraph of $H$ spanned by the blue edges is $(m-1)$ and so $H^b$ does not contain $K_m$ as a subgraph. Also, $H^r$ does not contain $F$, because in the components of size $(j_0-1)$ there is no red tree of size $j_0$ and in the component of size $(\sum_{j_0}^{n}-1)$ some trees of size $j_0$ or larger are missed. This means that we have a $(F,K_m)$-free coloring of $H$ and so, $H\nrightarrow(F,K_m)$. Therefore, $(j_0-1)(m-3)+ \sum_{j_0}^{n}$ is a lower bound for $R_*(F,K_m)$.

\bigskip
For the upper bound, let $G$ be a graph containing $K_{r-1}$ with an additional vertex $v$ adjacent to at least  $r_*$ vertices of $K_{r-1}$, $G=K_{r-1}\sqcup K_{1,r_*}$,  whose edges are arbitrarily colored red and blue. We suppose that $K_m\nsubseteq G^b$ and by   induction on the variety of $F$, $q(F)$, we prove that the subgraph of $G$ spanned by the red edges contains $F$ as a subgraph. If $q(F)=1$, then all trees in $F$ have the same number of vertices and so  $F=kT_n$, for some positive integers $k$ and $n$. Now, having applied Lemma \ref{starktreecomplete}, it is verified that
$R_*(F,K_m)=(n-1)(m-3)+nk,$
which coincides with the claimed number of the theorem.

\medskip
Now, let $F$ be an arbitrary forest contains trees of different sizes such that $q(F)\geq 2$. Let $s$ and $n$ be the sizes of the smallest and the largest components in $F$, respectively. Delete all trees of size $s$ from $F$ and let $F'$ be the resulting forest. Clearly, $F=F'\cup {k_sT_s}$ and if the size of the smallest tree in $F'$ is $t$, then  $F=\{k_sT_s, k_tT_t,\ldots ,k_nT_n\}$, $F'=\{k_tT_t,\ldots ,k_nT_n\}$ and $|V(F')|=\sum_{t}^{n}$. In the sequel, we prove that $G$ must contain a red copy of $F'$ as a subgraph. For this purpose, let $j_0$ be the smallest value of $j$ that realizes $\max_{j\in C(F')} \{(j-1)(m-2)+ \sum_{j}^{n}\}$. Note that by Theorem \ref{Rforest}, $R(F',K_m)=(j_0'-1)(m-2)+\sum_{j_0'}^{n}$, and since $C(F)\setminus C(F')=\{s\}$, we have either $j_0=s$ or $j_0=j_0'$, which means that $j_0\leq j_0'$. Set $r'_*=(j_0'-1)(m-3)+\sum_{j_0'}^{n}$. As $q(F')<q(F)$, by the induction hypothesis,  $R_*(F',K_m)\leq r'_*$. 

\medskip
Since $C(F')\subseteq C(F)$, then $(j_0-1)(m-2)+\sum_{j_0}^{n} \geq (j_0'-1)(m-2)+\sum_{j_0'}^{n}$, and using $j_0\leq j_0'$, we conclude that 

\begin{eqnarray*}
r_*=(j_0-1)(m-3)+\textstyle\sum_{j_0}^{n} &\geq & (j_0'-1)(m-2)+\textstyle\sum_{j_0'}^{n}-(j_0-1)\\
&\geq & (j_0'-1)(m-2)+\textstyle\sum_{j_0'}^{n}-(j_0'-1)\\
&= & (j_0'-1)(m-3)+\textstyle\sum_{j_0'}^{n}=r'_*.
\end{eqnarray*}
Therefore, $r'_* \leq r_*$ and so, $\deg(v)\geq r_*\geq r'_* \geq R_*(F',K_m)$. Note that $F'\subseteq F$ and  $R(F,K_m) \geq R(F',K_m)$. Set $r'=R(F',K_m)$ and choose $r'-1$ vertices from $G-\{v\}\simeq K_{r-1}$ containing  $r'_*$ vertices from $N(v)$ and let $G'$ be the subgraph of $G$ spanned by the vertex $v$ and the chosen $r'-1$ vertices. Clearly, $G'=K_{r'-1}\sqcup K_{1,r'_*}\subseteq G$ and since $r'_* \geq R_*(F',K_m)$, then   $G'\rightarrow(F',K_m)$. As $G'\subseteq G$ and $G$ does not contain a blue monochromatic copy of $K_m$, we obtain that  $F'\subseteq G^r$. Discard the vertices of such a copy of $F'$ from $G$ and let $H$ be the resulting graph. We prove that there is a red monochromatic copy of $k_sT_s$ in $H$. For this purpose, we consider the following cases.

\bigskip
\noindent{\bf Case 1.} Let $v\in V(F')$.

\medskip
\noindent In this case, $H\subseteq K_{r-1}$ and
 $|V(H)|= r-1-|V(F')\setminus\{v\}|\geq (s-1)(m-2)+sk_s.$
Therefore, $|V(H)| \geq (s-1)(m-2)+sk_s=R(k_sT_s,K_m)$ and so, $H\rightarrow(k_sT_s,K_m)$, which means that there is a  monochromatic red copy of $k_sT_s$ in $H$.

\bigskip
\noindent{\bf Case 2.} Let $v\not\in V(F')$.

\medskip
\noindent In this case, $V(F')\subseteq K_{r-1}$ and  $|V(H)\setminus \{v\}|=|V(K_{r-1})\setminus V(F')|$. Therefore,
\begin{eqnarray*}
|V(H)\setminus \{v\}| &=&  r-1-|V(F')| \\
& =& (j_0-1)(m-2)+\textstyle\sum_{j_0}^{n}-1-\textstyle\sum_{t}^{n}\\
&\geq & (s-1)(m-2)+\textstyle\sum_{s}^{n}-1-\textstyle\sum_{t}^{n}\\
&\geq & (s-1)(m-2)+sk_s-1 = R(k_sT_s,K_m)-1.
\end{eqnarray*}

\bigskip
In addition,
\begin{eqnarray*}
\deg_H(v) \geq r_* -|V(F')| &=& (j_0-1)(m-3)+\textstyle\sum_{j_0}^{n}-\textstyle\sum_{t}^{n}\\
&\geq & (s-1)(m-3)+\textstyle\sum_{s}^{n}-\textstyle\sum_{t}^{n}\\
&=& (s-1)(m-3)+sk_s=R_*(k_sT_s,K_m).
\end{eqnarray*}

Therefore, in this case $|V(H)\setminus \{v\}|\geq R(k_sT_s,K_m)-1$ and $\deg_H(v) \geq R_*(k_sT_s,K_m)$. Having applied Lemma \ref{starktreecomplete}, we conclude that $H\rightarrow (k_sT_s,K_m)$. Since there is no blue $K_m$ in $G$, then $H$ contains a red monochromatic copy of $k_sT_s$, all trees of size $s$ in $F$. Now, this copy of $k_sT_s\subseteq H^r$ with the deleted monochromatic red forest $F'$, form a monochromatic red copy of $F$ in $G$, completing the proof of the theorem.
\end{proof}

\bigskip
\section{Multicolor Results}

\medskip
In this section, the star-critical Ramsey number of an arbitrary forest versus any number of complete graphs will be computed exactly. Before that, we prove the following lemma determining the exact value of the Ramsey number of an arbitrary forest versus any number of complete graphs, in terms of the Ramsey number of complete graphs.

\medskip
\begin{lemma}\label{Rcompletes}
Let $F$ be an arbitrary forest, $m_1,m_2,\ldots,m_t$ be positive integers and let $r=R(K_{m_1},\ldots,K_{m_t})$. Then,
$R(F,K_{m_1},\ldots,K_{m_t})=\max_{j\in C(F)} \{(j-1)(r-2)+ \sum_{i=j}^{n(F)} ik_i(F) \}$.
\end{lemma}

\begin{proof}
Let the maximum for the claimed number occur for $j=j_0$ and $p_0=\sum_{i=j_0}^{n(F)} ik_i(F)$. For the lower bound, begin with a $t$-coloring of the edges of $K_{r-1}$, say colors $\alpha_1,\ldots,\alpha_t$, that has no copy of $K_{m_i}$ in color $\alpha_i$ for any $i$, $1\leq i\leq t$. Consider $r-2$ vertices of $K_{r-1}$ and replace each of them by a complete graph of order $j_0-1$ and replace the remaining vertex of $K_{r-1}$ by a complete graph of order $p_0-1$ and let $G$ be the resulting graph. Color all  edges contained in these complete graphs  by color $\alpha_0$.  Each edge in the original graph $K_{r-1}$ expands into a copy of $K_{j_0-1,j_0-1}$ or  a copy of $K_{j_0-1,p_0-1}$ in $G$. Color all edges in these subgraphs with the same color that its original edge in $K_{r-1}$ had. Note that the subgraph of $G$ induced by the edges of color $\alpha_0$ does not contain $F$, because the parts have sizes $(j_0-1)$ and $(p_0-1)$ and in the components of size $(j_0-1)$ there is no tree of size $j_0$ and in the component of size $(p_0-1)$ some trees of size $j_0$ or larger are missed. This yields a $(t+1)$-edge coloring of $K_{(j_0-1)(r-2)+p_0-1}$ with colors $\alpha_0, \alpha_1,\ldots,\alpha_t$ without a monochromatic copy of $F$ in color $\alpha_0$ and no monochromatic copy of $K_{m_i}$ in color $\alpha_i$, $1\leq i\leq t$. Therefore, $K_{(j_0-1)(r-2)+p_0-1}\nrightarrow (F,K_{m_1},\ldots,K_{m_t})$.

\medskip
For the upper bound, let $s$ be the claimed number for $R(F,K_{m_1},\ldots,K_{m_t})$ and consider an arbitrary $(t+1)$-edge coloring $c$ of the complete graph $G=K_s$ by colors $\alpha_0, \alpha_1,\ldots,\alpha_t$. Recolor all edges of colors $\alpha_1, \alpha_2,\ldots,\alpha_t$ by a new color $\beta_0$ and retain the color of the remaining edges. This yields an edge coloring of $K_s$ by colors $\alpha_0$ and $\beta_0$. By Theorem \ref{Rforest}, $R(F,K_{r})=s$ and so $K_s$ contains a monochromatic copy of $F$ in color $\alpha_0$ or a monochromatic copy of $K_r$ in color $\beta_0$. If the first case occurs, we are done, and otherwise we have a monochromatic copy of $K_r$ in color $\beta_0$. Return to $c$, restricted to this set of $r$ vertices. By the definition of the Ramsey number, $c$ has a monochromatic copy of $K_{m_i}$ in color $\alpha_i$, for some $i$, $1\leq i\leq t$. Hence, $K_s\rightarrow(F,K_{m_1},\ldots,K_{m_t})$ and the proof is completed.
\end{proof}

\medskip
\begin{proposition}\label{ramseyfull}
For given positive integers $m_1,m_2,\ldots,m_t$, the $t$-tuple $(K_{m_1},\ldots,K_{m_t})$ of complete graphs  is Ramsey-full.
\end{proposition}

\begin{proof}
Let $r=R(K_{m_1},\ldots,K_{m_t})$ and $e=uv$ be an arbitrary edge of $K_r$. We prove that $K_r-e\nrightarrow(K_{m_1},\ldots,K_{m_t})$. On the contrary, let $K_r-e\rightarrow(K_{m_1},\ldots,K_{m_t})$. This means that in any $t$-coloring of the edges of $K_r-e$ there exists some $i$, $1\leq i\leq t$, such that the $i$-th color class contains $K_{m_i}$ as a subgraph. Since $u$ and $v$ are not adjacent, such a copy of $K_{m_i}$ cannot contain $u$ and $v$ at the same time and so, such a copy would appear in the induced $t$-edge coloring of either $K_{r}\setminus\{u\}$ or $K_{r}\setminus\{v\}$. Therefore, $K_{r}\setminus\{u\}\rightarrow(K_{m_1},\ldots,K_{m_t})$ or $K_{r}\setminus\{v\}\rightarrow(K_{m_1},\ldots,K_{m_t})$, which means that $K_{r-1}\rightarrow(K_{m_1},\ldots,K_{m_t})$, a contradiction.
\end{proof}

\medskip
By Theorem \ref{Rforest} and Lemma \ref{Rcompletes}, $R(F,K_{m_1},\ldots,K_{m_t})=R(F,K_r)$ and so by Theorem \ref{starforestcomplete}, we obtain the following theorem which determines the exact value of the star-critical Ramsey number of a forest versus complete graphs of arbitrary sizes. Note that  if $H$ is an arbitrary graph whose vertices are partitioned into sets $V_1,V_2,\ldots,V_t$, then the {\it shadow graph} of $H$, denoted by $\Gamma(H)$, is a $t$-vertex graph with vertices $v_1,v_2,\ldots,v_t$ such that $v_iv_j\in E(\Gamma(H))$ if and only if $E[V_i,V_j]\neq\emptyset$.

\medskip
\begin{theorem}
Let $F$ be an arbitrary forest without isolated vertices, $m_1, \ldots,m_t$ be positive integers and $r=R(K_{m_1},K_{m_2},\ldots,K_{m_t})$. If $j_0$ is the smallest value of $j$ that realizes ${\max_{j\in C(F)}} \{(j-1)(r-2)+ \sum_{i=j}^{n(F)}ik_i(F)\}$, then,
$$R_*(F, K_{m_1},\ldots,K_{m_t})=(j_0-1)(r-3)+\sum_{i=j_0}^{n(F)}ik_i(F).$$
\end{theorem}

\begin{proof}
Let $r_*$ be the claimed number for $R_*(F, K_{m_1},\ldots,K_{m_t})$. Note that by Lemma \ref{Rcompletes}, $R(F, K_r)=R(F, K_{m_1},\ldots,K_{m_t})=(j_0-1)(r-2)+ \sum_{j_0}^{n}$ and by Theorem \ref{starforestcomplete}, $R_*(F,K_r)=r_*$. 

\medskip
For the lower bound, we prove that if $H=K_{R(F,K_r)-1}\sqcup K_{1,r_*-1}$, then $H\nrightarrow(F, K_{m_1},\ldots,K_{m_t})$.  For this purpose, partition the vertices of $K_{R(F,K_r)-1}$ into $(r-1)$ parts $V_1,\ldots,V_{r-1}$ such that for every $i$, $1\leq i\leq r-2$, $|V_i|=j_0-1$ and $|V_{r-1}|=\sum_{j_0}^{n}-1$. Color all edges with both ends in $V_i$, $1\leq i\leq r-1$, by $\alpha_0$ and the rest by $\beta_0$. Now, add a vertex $v$  adjacent to every vertex in $V_i$, $2\leq i\leq r-1$, by color $\beta_0$. Clearly,  we have a 2-coloring of the edges of $H$ with colors $\alpha_0$ and $\beta_0$ such that the subgraph spanned by the edges of color $\alpha_0$, $H^{\alpha_0}$, does not contain $F$ and the subgraph spanned by the edges of color $\beta_0$, $H^{\beta_0}$, does not contain $K_r$ as a subgraph. Note that $H^{\beta_0}$ is an $r$-partite subgraph of $H$ with partite sets $(\{v\},V_1,V_2,\ldots,V_{r-1})$, which is not a complete $r$-partite graph. Let $\Gamma(H^{\beta_0})$ be the shadow graph of $H^{\beta_0}$. Clearly, $\Gamma(H^{\beta_0})$ is a non-complete graph on $r$ vertices and so, by Proposition \ref{ramseyfull}, $\Gamma(H^{\beta_0})$ has a $(K_{m_1},\ldots,K_{m_t})$-free coloring. Consider such a  $(K_{m_1},\ldots,K_{m_t})$-free coloring of $\Gamma(H^{\beta_0})$  with colors $\alpha_1, \alpha_2,\ldots,\alpha_t$. Now, each edge of $\Gamma(H^{\beta_0})$ in the original graph $H^{\beta_0}$ expands into a complete bipartite subgraph. Recolor all edges in these subgraphs with the  color that its original edge in $\Gamma(H^{\beta_0})$ had. This yields a $(F,K_{m_1},\ldots,K_{m_t})$-free coloring of  $H$ with $t+1$ colors $\alpha_0,\alpha_1, \alpha_2,\ldots,\alpha_t$.

\bigskip
For the upper bound, let $c$ be an arbitrary $(t+1)$-coloring of the edges of $G=K_{R(F,K_r)-1}\sqcup K_{1,r_*}$ by colors $\alpha_0,\alpha_1,\ldots,\alpha_t$. We prove that $G\rightarrow (F, K_{m_1},\ldots,K_{m_t})$. Recolor all edges of colors $\alpha_1, \alpha_2,\ldots,\alpha_t$ by a new color $\beta_0$ and retain the color of the remaining edges. This yields a 2-edge coloring of  $G=K_{R(F,K_r)-1}\sqcup K_{1,r_*}$ by colors $\alpha_0$ and $\beta_0$ and so, by the definition, $G$ contains either a monochromatic copy of $F$ in color $\alpha_0$ or a monochromatic  copy of $K_r$ in color $\beta_0$. If we have a monochromatic copy of $K_r$ in color $\beta_0$, then return to $c$, restricted to this set of $r$ vertices. By the definition of the Ramsey number, the coloring induced by $c$ to this set of $r$ vertices has a monochromatic copy of $K_{m_i}$ in color $\alpha_i$, for some $i$, $1\leq i\leq t$. Therefore, $G\rightarrow(F,K_{m_1},\ldots,K_{m_t})$ and the proof is completed.
\end{proof}

\medskip
Let $G_1, G_2, \ldots , G_t$ be given graphs, $r=R(G_1, G_2, \ldots, G_t)$ and $r_{*}=R_{*}(G_1, G_2, \ldots, G_t)$. By the definition of the star-critical Ramsey number, $G=K_{r-1}\sqcup K_{1,r_{*}}$ is a graph of size  ${r-1 \choose 2}+r_{*}$ such that $G\rightarrow(G_1, G_2, \ldots, G_t)$. Therefore,  for given graphs $G_1, G_2, \ldots , G_t$, it can be easily seen that  $$\hat{R}(G_1, G_2, \ldots , G_t) \leq {R(G_1, G_2, \ldots, G_t)-1 \choose 2}+R_*(G_1, G_2, \ldots, G_t).$$

Therefore, if $F$ is an arbitrary forest and $m\geq 3$, then

\begin{equation}\label{equ2}
\hat{R}(F, K_{m}) \leq {R(F, K_{m})-1 \choose 2}+ R_*(F,K_m).
\end{equation}

Let $j_0$ be the smallest value of $j$ that realizes $\max_{j\in C(F)} \{(j-1)(m-2)+ \sum_{i=j}^{n(F)}ik_i(F)\}$. By Theorem \ref{starforestcomplete} and  (\ref{equ2}), we deduce that

\begin{equation}\label{equ}
\hat{R}(F, K_{m}) \leq {R(F, K_{m}) \choose 2}- (j_0-2).
\end{equation}

In particular, if $F$ is a matching of size $t$, i.e. $F=tK_2$, then by Theorem \ref{Rforest}, ${R}(tK_2, K_{m})={m+2t-2}$ and so by (\ref{equ}),  $\hat{R}(tK_2, K_{m})\leq{m+2t-2 \choose 2}.$ Erd\H{o}s and Faudree in \cite{s10} proved that if $m$ and $t$ are positive integers and $m\geq 4t-1$, then $\hat{R}(tK_2, K_{m})={m+2t-2 \choose 2}.$ Thus, if $F$ is a matching of size $t$, then equality holds in (\ref{equ}),  which means that the bound presented in (\ref{equ}) is best possible.

\section*{Conflict of interest}
The authors have no conflicts of interest to declare. All  authors  have  seen  and  agree  with  the  contents  of  the  manuscript  and there  is  no  financial  interest  to  report.  We  certify  that  the  submission  is original work and is not under review at any other publication.

\end{document}